\documentclass{amsart}
\usepackage{amsmath}
\usepackage{amsfonts}
\usepackage{amssymb}
\usepackage{ascmac}
\usepackage{layout}
\usepackage{amsthm}

\makeatletter
\@namedef{subjclassname@2020}{%
  \textup{2020} Mathematics Subject Classification}
\makeatother

\newcommand{\pcf}[1]{\operatorname{pcf}(#1)}
\newcommand{\cf}{\operatorname{cf}}
\newcommand{\force}{\Vdash}
\newcommand{\reg}{{\rm Reg}}

\title{Prikry-type Forcing and the Set of Possible Cofinalities}

\author{Kenta Tsukuura}
\address{Doctoral Program in Mathematics, Degree Programs in Pure and Applied Sciences, Graduate School of Science and Technology, University of Tsukuba, Tsukuba, 305-8571, Japan}
\email{tukuura@math.tsukuba.ac.jp}

\thanks{This research was supported by Grant-in-Aid for JSPS Research Fellow Number 20J21103 and JSPS KAKENHI Grant Nos.18K03403 and 18K03404. The author is grateful to Masahiro Shioya for helpful discussions.}

\theoremstyle{plain}
\newtheorem{thm}{Theorem}[section]
\newtheorem{defi}[thm]{Definition}
\newtheorem{lem}[thm]{Lemma}
\newtheorem{coro}[thm]{Corollary}
\newtheorem{ques}[thm]{Question}

\begin{document}
\maketitle 
\begin{abstract}
It is known that the set of possible cofinalities $\pcf{A}$ has good properties if $A$ is a progressive interval of regular cardinals. In this paper, we give an interval of regular cardinals $A$ such that $\pcf{A}$ has no good properties in the presence of a measurable cardinal, or in generic extensions by Prikry-type forcing.

\end{abstract}

\section{Introduction}
Cardinal arithmetic has been one of the most important areas in set theory. 
Shortly after Cohen devised the method of forcing, 
Easton~\cite{easton} proved that the powers of regular cardinals 
is subject only to K\"onig's theorem in ZFC. 
Easton's theorem left the behavior of the powers of \emph{singular} cardinals 
as the Singular Cardinal Problem. 
Some time later, Silver~\cite{silver} proved 
the first nontrivial result around the problem:
Singular cardinals of uncountable cofinality cannot be the least cardinal 
at which the Generalized Continuum Hypothesis (GCH) fails. 
Still later, 
Shelah~\cite{shelah} developed pcf theory and established a result 
that supersedes Silver's theorem:

      \begin{thm}[Shelah]\label{shelah}
    $\aleph_{\omega}^{\aleph_{0}} < (2^{\aleph_{0}})^{+} + \aleph_{\omega_{4}}$.
      \end{thm}
      
An outline of the proof is as follows.
First we have
$\aleph_{\omega}^{\aleph_{0}} = 2^{\aleph_{0}} + {\rm cf}([\aleph_{\omega}]^{\aleph_0},\subseteq)$.
The crucial claim is that
${\rm cf}([\aleph_{\omega}]^{\aleph_0},\subseteq)
=\max{\pcf{\{\aleph_{n} \mid n < \omega\}}}<\aleph_{\omega_4}$.
Shelah proved it by analyzing the structure of 
$\pcf{\{\aleph_{n} \mid n < \omega\}}$.
More specifically, he obtained the latter inequality 
by showing the following results
for a progressive interval of regular cardinals $A$:
\begin{itemize}
 \item $\pcf{A}$ is an interval of regular cardinals with a largest element.
 \item $|\pcf{A}| < |A|^{+4}$.
\end{itemize}

Theorem~\ref{shelah} can be generalized for a \emph{non}-fixed point 
of the $\aleph$ function. 
Let $\kappa$ be a singular cardinal with $\kappa = \aleph_\mu > \mu$. 
Shelah proved that
${\rm cf}([\kappa]^{|\mu|},\subseteq)=\max \pcf{A}$ for some progressive interval of regular cardinals $A$ with $\sup A = \kappa$. 
As before, this reduces the investigation of the power of $\kappa$ 
to that of the structure of 
${\pcf{A}}$. 
Note that we can take $A$ to be progressive because
$\kappa$ is a \emph{non}-fixed point of the $\aleph$ function. 
Thus the assumption of $A$ being progressive seems essential in pcf theory. 
Now one may ask
     \begin{ques}\label{question}
What if $A$ is a {non}-progressive interval of regular cardinals?
     \end{ques}
Motivated by the question, we prove in this paper
     \begin{thm}\label{maintheorem}
    Suppose $\kappa$ is a measurable cardinal. Then the following hold:
     \begin{enumerate}
      \item[$(1)$] $\pcf{\kappa \cap \reg} = (2^{\kappa})^{+}\cap \reg$.
      \item[$(2)$] Prikry forcing over $\kappa$ forces that $\pcf{{\kappa} \cap {\rm Reg}} = (2^{{\kappa}})^{+} \cap {\rm Reg}$.
      \end{enumerate}
     \end{thm}
     
From Theorem~\ref{maintheorem}~(1) we get

     \begin{coro}\label{coro1}
         Suppose $\kappa$ is a measurable cardinal. Then the following hold:
     \begin{enumerate}
      \item[$(1)$] 
      $\pcf{\kappa \cap \reg}$ has no largest element if $2^{\kappa}$ is singular.
      \item[$(2)$] 
      $|\pcf{\kappa \cap \reg}|> |\kappa \cap \reg|^{+4}$ if $2^{\kappa} > \kappa^{+(\kappa^{+4})}$.
      \end{enumerate}
     \end{coro}
     
The following corollary of Theorem~\ref{maintheorem}~(2) answers Question~\ref{question}:

     \begin{coro}\label{coro2}
          Suppose there is a supercompact cardinal. Then in some forcing extension
          there is a {non}-progressive interval of regular cardinals $A$
          such that ${\rm sup}(A)$ is singular, $\pcf{A}$ has no largest element 
          and $|\pcf{A}|>|A|^{+4}$.
     \end{coro}

The proof of Corollary~\ref{coro2} is as follows.
Let $\kappa$ be a supercompact cardinal.
We may assume that $\kappa$ is indestructibly supercompact 
in the sense of Laver~\cite{laver}. 
This enables us to get a model in which $\kappa$ is supercompact and
$2^{\kappa}$ is a singular cardinal $> \kappa^{+(\kappa^{+4})}$.
Finally,
Prikry forcing gives a model in which $A=\kappa \cap \reg$ is as desired
by Theorem~\ref{maintheorem}~(2).
See Corollary~\ref{coro34} for an additional property of the final model.

Some large cardinal hypothesis is necessary in Theorem~\ref{maintheorem}.
Assume on the contrary that
GCH holds and there is no weakly inaccessible cardinal.
A simple argument shows that if $A \subseteq \reg$, then 
$\pcf{A} = A \cup \{(\sup B)^{+}\mid B \subseteq A$ has no maximal element$\}$,
so that
$\pcf{A}$ has a largest element and $|\pcf{A}|<|A|^{+4}$.
      
The structure of this paper is as follows. 
In Section~2, we recall basic facts of pcf theory and Prikry forcing. 
Theorem~\ref{maintheorem} is proved in Section~3. 
We also consider the problem whether $\pcf{\pcf{A}}=\pcf{A}$ holds. 
In Section~4, we prove an analogue of Theorem~\ref{maintheorem}~(2) 
for Magidor forcing.

    \section{Preliminaries}
    In this section, we recall basic facts of pcf theory and Prikry forcing. 
    For more on the topics, we refer the reader to \cite{AM} and \cite{gitik} respectively.
    We also use \cite{kanamori} as a reference for set theory in general.
    
  Our notation is standard. We let $\reg$ denote the class of all regular cardinals.
  Let $A \subseteq \reg$. 
  Then $\prod A$ is the set 
  $\{f:A \to \bigcup A \mid \forall \gamma \in A(f(\gamma)< \gamma )\}$. 
 Let $F$ be a filter over $A$. We define a strict order $<_{F}$ on $\prod A$ 
  by $f <_{F} g$ iff 
   $\{\gamma \in A \mid f(\gamma) < g(\gamma)\}\in F$.
    
 \begin{defi}
 For $A\subseteq \reg$,
$$\pcf{A} = \left\{\cf\left(\prod A,<_{D}\right) \mid D\text{ is an ultrafilter over }A
\right\}.$$
 \end{defi}
 
Note that
$A \subseteq \pcf{A} \subseteq (2^{\sup A})^{+}\cap \reg$.
 If there is an increasing and cofinal sequence in 
 ${\left(\prod A,<_{F}\right)}$ of length $\theta$ for some filter $F$ over $A$, 
 then $\cf(\theta) \in \pcf{A}$. 
 
  A set $A\subseteq \reg$ is \emph{progressive} if $\min A > |A|$. 
  An interval of regular cardinals is a set of the form 
  $[\lambda,\kappa)\cap \reg$
  for a pair of cardinals $\lambda<\kappa$.
  Here is the fundamental theorem on progressive intervals of regular cardinals.
 \begin{thm}[Shelah]\label{progressive}
 If $A\subseteq \reg$ is a progressive interval, then we have
  \begin{enumerate}
   \item[$(1)$] $\pcf{A}$ has a largest element.
   \item[$(2)$] $|\pcf{A}| < |A|^{+4}$.
   \item[$(3)$] $\pcf{\pcf{A}} = \pcf{A}$.
\end{enumerate}
 \end{thm}

 Theorem~\ref{scaleexists} is known as the scale theorem. 
 \begin{thm}[Shelah]\label{scaleexists}
 Suppose $\kappa$ is a singular cardinal. 
  Then there is a set 
  $A \in [\kappa\cap \reg]^{\cf(\kappa)}$ such that $\sup A = \kappa$ and ${\left(\prod A,<_F\right)}$ has an increasing and cofinal sequence of length $\kappa^{+}$. Here, $F$ is the cobounded filter over $A$.
  In particular, $\kappa^{+} \in \pcf{\kappa \cap \reg}$. 
 \end{thm}
 
 Next, we recall basic facts of Prikry forcing from \cite{prikry}. 
 Let $\kappa$ be a measurable cardinal and $U$ a normal ultrafilter over $\kappa$. Prikry forcing $\mathbb{P}$ is the set $[\kappa]^{<\omega} \times U$ ordered by 
 ${\langle b,Y\rangle}\leq {\langle a,X\rangle}$ iff $a\subseteq_{e}b$ 
 (i.e. $a=b \cap (\max(a)+1)$), $Y \subseteq X$ and $b \setminus a \subseteq X$.

   $\mathbb{P}$ has the $\kappa^{+}$-c.c. and size $2^{\kappa}$. 
   Thus, $\mathbb{P}$ does not change the value of $2^{\theta}$ 
   for any ${\theta}\ge\kappa$. 
   $\mathbb{P}$ preserves all cardinals above ${\kappa}$
   but changes the cofinality of ${\kappa}$.
   Let $\dot{g}$ be a $\mathbb{P}$-name such that 
   $\mathbb{P}\force \dot{g} = \bigcup \{ a \mid \exists X({\langle a,X\rangle} \in \dot{G})\}$, where $\dot{G}$ is the canonical $\mathbb{P}$-name for a generic filter. 
   Then $\dot{g}$ is forced to be a cofinal subset of $\kappa$ of order type $\omega$.
   Moreover, we need
  \begin{itemize}
   \item 
   ${\langle a,X\rangle} \force {a} \subseteq_{e}\dot{g} \land \dot{g}\setminus{a} \subseteq {X}$.
   In particular, $\mathbb{P}\force \dot{g} \subseteq^{*} X$ for every $X\in U$.
   \item If $\langle a,X \rangle \force \xi \in \dot{g}$, then $\xi \in a$.
  \end{itemize} 
   The latter property follows by $\langle a,X\setminus (\xi + 1) \rangle \leq \langle a,X\rangle$ forces $\dot{g} \setminus a \subseteq X \setminus (\xi + 1)$. 

  For subsequent purposes, we present a direct proof of Prikry lemma. 
  Suppose $\{X_{b}\mid b \in [\kappa]^{<\omega}\}\subseteq U$. 
  The diagonal intersection $\triangle_{b}X_{b}$ is defined to be the set
  $\{\xi < \kappa \mid \forall b \in [\xi]^{<\omega}(\xi \in X_{b})\}$. 
  Since $U$ is normal, we have $\triangle_{b}X_{b} \in U$.

  \begin{lem}\label{diagonallemma}
   Suppose $\{X_{b} \mid b \in [\kappa]^{<\omega}\} \subseteq U$ 
   and $a \in [\kappa]^{<\omega}$. 
   Then any extension of ${\langle a,\triangle_{b}X_{b}\rangle}$ is compatible with ${\langle a,X_{a}\rangle}$.
  \end{lem}
  \begin{proof}
   Let
   ${\langle c,Y\rangle}\leq {\langle a,\triangle_{b}X_{b}\rangle}$. 
   Then $c\setminus a\subseteq X_{a}$ by $a\subseteq_{e} c$ and
   $c\setminus a\subseteq \triangle_{b}X_{b}$. 
   Thus ${\langle c,Y\cap X_{a}\rangle}$ is a common extension of 
   ${\langle c,Y\rangle}$ and ${\langle a,X_{a}\rangle}$, as desired.
  \end{proof}
  
 \begin{lem}[Prikry lemma]\label{prikrycondition}
  Let $a \in [\kappa]^{<\omega}$ and $\sigma$ 
  be a statement of the forcing language. 
  Then there is an $X \in U$ such that ${\langle a,X\rangle}$ decides $\sigma$, i.e. ${\langle a,X\rangle }\force \sigma $ or ${\langle a,X\rangle}\force \lnot \sigma$.
 \end{lem}
 \begin{proof}

  For each $b \in [\kappa]^{<\omega}$ define $X_{b} \in U$ as follows:
  If $a\subseteq_{e} b$, let $X_{b}$ be the unique set from the following 
  mutually disjoint sets
\begin{itemize}
 \item $X_{b}^{+} = \{\xi < \kappa \mid {b} \subseteq \xi \land
 \exists Y\in U ({\langle b\cup\{\xi\} ,Y \rangle} \force \sigma) \}$.
 \item $X_{b}^{-} = \{\xi < \kappa \mid {b} \subseteq \xi \land
 \exists Y\in U({\langle b\cup\{\xi\} ,Y \rangle} \force \lnot\sigma)\}$.
 \item $X_{b}^{0} = \kappa \setminus (X_{b}^{+} \cup X_{b}^{-})$.
\end{itemize}
  Otherwise, let $X_{b}= \kappa$.
For each $b \in [\kappa]^{<\omega}$ define $Y_{b} \in U$ as follows:
If there is a $Y \in U$ such that ${\langle b,Y\rangle}$ decides $\sigma$,
let $Y_{b}$ be one such $Y$. Otherwise, let $Y_{b}= \kappa$. 
We claim that $X= \triangle_{b}(X_{b}\cap Y_{b})\in U$ is as desired. 
Take an arbitrary extension ${\langle c,Y\rangle}\leq {\langle a,X\rangle}$ that decides $\sigma$. 
We may assume $c = b\cup\{\xi\}$ with $a \subseteq_{e} b \subseteq \xi$.
Note that
${\langle c,Y_{c}\rangle}$ decides $\sigma$.
We may assume ${\langle c,Y_{c}\rangle } \force \sigma$.
Then ${\langle c,\triangle_{b}Y_{b}\rangle}\force \sigma$
by Lemma~\ref{diagonallemma}. 
Thus ${\langle c,X \rangle} \le {\langle c,\triangle_{b}Y_{b}\rangle}$
forces $\sigma$.
We claim that ${\langle b,X\rangle} \force \sigma$,
which completes the proof by repeating the argument.

It suffices to show that any extension of ${\langle b,X\rangle}$ 
is compatible with a condition forcing $\sigma$.
Let ${\langle d,Z\rangle} \leq {\langle b,X\rangle}$.
We may assume $b \subsetneq_{e} d$.
Note that
$\xi \in X_{b}$ by $\xi \in X$, and hence $X_{b} = X_{b}^{+}$
by ${\langle b\cup\{\xi\},X\rangle}\force \sigma$. 
Let $\eta = \min(d\setminus b) \in X$.
Then $\eta \in X_{b} = X_{b}^{+}$, so
${\langle b\cup\{\eta\},Y\rangle}\force\sigma$
for some $Y$, and hence
${\langle b\cup\{\eta\},Y_{b\cup\{\eta\}}\rangle}\force\sigma$.
Note that
$d\setminus (b\cup\{\eta\}) \subseteq Y_{b\cup\{\eta\}}$ by
$b\cup\{\eta\} \subseteq_{e} d$ and
$d\setminus (b\cup\{\eta\}) \subseteq d\setminus b \subseteq X$.
Thus ${\langle d, Z\cap Y_{b\cup\{\eta\}}\rangle}$ 
is a common extension of
${\langle d,Z\rangle}$ and ${\langle b\cup\{\eta\},Y_{b\cup\{\eta\}}\rangle}$,
as desired.
  \end{proof}

 \begin{coro}
  $\mathbb{P}$ adds no new bounded subsets of $\kappa$. 
  In particular, $\mathbb{P}$ preserves all cardinals below $\kappa$.
 \end{coro}
 
   \section{Prikry Forcing and a Non-progressive Interval}
   The first half of this section is devoted to
   
   \begin{proof}[Proof of Theorem~\ref{maintheorem}]
  Let $\kappa$ be a measurable cardinal. 
  Take a normal ultrafilter $U$ over $\kappa$ and form
  $j:V \to M \simeq \text{Ult}(V,U)$.
  For each $\alpha\le2^{\kappa}$, we can choose 
    $f_{\alpha} \in \mbox{}^{\kappa}\kappa$ such that $\alpha =[f_{\alpha}]_{U}$
   by $2^{\kappa}\le(2^{\kappa})^{M}<j(\kappa)$.

    Note that
   ${\kappa \cap \reg}\subseteq \pcf{\kappa \cap \reg}\subseteq (2^{\kappa})^{+}\cap \reg$. To complete the proof, it suffices to show that 
   $[\kappa,(2^{\kappa})^{+})\cap\reg\subseteq \pcf{\kappa\cap\reg}$ in both cases, (1) and (2).

    (1) Let $\theta \in [\kappa,(2^{\kappa})^{+})\cap\reg$.
    Then we may assume $f_{\theta} \in \mbox{}^{\kappa}(\kappa\cap\reg)$.
    Since $\kappa=[{\rm id}]_{U}\le[f_{\theta}]_{U}$, we have
   $$X = \{\xi < \kappa\mid \forall \eta < \xi( f_{\theta}(\eta) < \xi)
   \land \xi \leq f_{\theta}(\xi)\} \in U.$$
   Note that
   $f_{\theta}\upharpoonright X$ is strictly increasing.
   Define an ultrafilter $U_{\theta}$ over $\kappa \cap \reg$ by 
   $Y \in U_{\theta}$ iff $f_{\theta}^{-1}``Y \in U$. 
Then we have
   ${\left(\prod_{\xi\in X}f_{\theta}(\xi),<_{U}\right)}\simeq
   {\left(\prod f_{\theta}``{X},<_{U_{\theta}}\right)}\simeq
   {\left(\prod \kappa\cap \reg,<_{U_{\theta}}\right)}$.
\medskip    

Since $\langle{f_{\alpha}\upharpoonright X \mid \alpha < \theta}\rangle$ is increasing and cofinal in ${\left(\prod_{\xi \in X}f_{\theta}(\xi),<_{U}\right)}$, 
we have
   $\theta =\cf\left(\prod_{\xi \in X}f_{\theta}(\xi),<_{U}\right)
   =\cf(\prod \kappa \cap \reg ,<_{U_{\theta}})\in \pcf{\kappa\cap \reg}$, as desired. 
\medskip
 
  (2)
  Let $\mathbb{P}$ be Prikry forcing defined by $U$.
  Note that the set $(\kappa,(2^{\kappa})^{+})\cap \reg$
  remains the same after forcing with $\mathbb{P}$ and $\kappa$ is singular.
   Let $\theta\in(\kappa,(2^{\kappa})^{+})\cap\reg$. 
  It suffices to prove that 
   $\mathbb{P}\force\theta\in \pcf{\kappa\cap\reg}$. Again, we may assume $f_\theta \in {^{\kappa}(\kappa \cap \reg)}$.
First, note that
      $$X = \{\xi < \kappa\mid \forall \eta < \xi( f_{\theta}(\eta) < \xi)
   \land \xi< f_{\theta}(\xi)\} \in U.$$
    Since $\mathbb{P} \force \dot{g}\subseteq^{*} {X}$, 
    we have
    \begin{center}
     $\mathbb{P} \force {\left(\prod_{\xi \in \dot{g}}{f}_{\theta}(\xi),<^{*}\right)} \simeq {\left(\prod {f}_{\theta}{``}\dot{g},<_{\dot{F}}\right)}.$    \end{center} 
     Here $<^{*}$ and $\dot{F}$ are $\mathbb{P}$-names 
     for the order on $\prod_{\xi \in \dot{g}}{f}_{\theta}(\xi)$
 defined by the cobounded filter over $\dot{g}$, and
     the cobounded filter over ${f}_{\theta}{``}\dot{g}$ respectively. 
     Thus it suffices to prove
 
        \begin{itemize}
    \item[(i)]
    $\mathbb{P}\force{\langle{f}_{\alpha}\upharpoonright \dot{g} \mid \alpha < {\theta}\rangle}${ is increasing in }${\left(\prod_{\xi \in \dot{g}} {f}_{\theta}(\xi),<^{*} \right)}$. 
    \item[(ii)]
    $\mathbb{P}\force{\langle{f}_{\alpha}\upharpoonright \dot{g} \mid \alpha < {\theta}\rangle}${ is cofinal in }${\left(\prod_{\xi \in \dot{g}} {f}_{\theta}(\xi),<^{*} \right)}$. 
   \end{itemize}

(i)
Let $\alpha < \beta$.
Then $Y= \{\xi < \kappa \mid f_{\alpha}(\xi) < f_{\beta}(\xi)\}\in U$. 
If ${\langle a,Z\rangle} \in \mathbb{P}$, then
${\langle a,Y \cap Z\rangle} \force\forall \xi\in\dot{g} \setminus{a} 
({f}_{\alpha} (\xi)<{f}_{\beta}(\xi))$, as desired. 
     
(ii)
   By the proof of (i), it suffices to show that 
   $\left\{h \upharpoonright \dot{g}\mid h\in \prod^{V}_{\xi \in{X}}{f}_{\theta}(\xi) \right\}$ 
is forced to be cofinal in 
${\left(\prod_{\xi \in \dot{g}} {f}_{\theta}(\xi),<^{*} \right)}$. 

Assume
$\force \dot{h} \in \prod_{\xi \in \dot{g}}{f}_{\theta}(\xi)$.
   For each $b\in [\kappa]^{<\omega}$
   define $Y_{b}\in U$ and $\eta_b < \kappa$ as follows.
   Note that ${\langle b, X\rangle}$ forces
   $b \subseteq\dot{g}$ and hence $\dot{h}(\max{{b}}) < f_{\theta}(\max{b})$.
     By Prikry lemma, there is a ${\langle b,Y_{b}\rangle}\le{\langle b, X\rangle}$ 
     that decides 
     $\dot{h}(\max{{b}}) ={\eta}$ for every $\eta < f_{\theta}(\max{b})$.
Then we can take an $\eta_b < f_{\theta}(\max{b})$ such that 
$${\langle b,Y_{b}\rangle}\force \dot{h}(\max{{b}})=\eta_b.$$
For each $\xi \in{X}$ define
     $$h(\xi) = \sup \{\eta_b+1 \mid b\in [\xi + 1]^{<\omega}\}.$$ 
Since ${f}_{\theta}(\xi)> \xi$ is regular,       
we have $h\in \prod_{\xi \in{X}}{f}_{\theta}(\xi)$ in $V$.
Let $Y = \triangle_{b}Y_{b}\in U$.
We claim that 
${\langle a, Y\rangle}\force \forall \xi\in\dot{g}\setminus a(\dot{h}(\xi)<{h}(\xi))$
for every $a\in [\kappa]^{<\omega}$,
which completes the proof.
It suffices to show that any extension of ${\langle a, Y\rangle}$
forcing $\xi\in\dot{g}\setminus a$ is compatible with a condition forcing
$\dot{h}(\xi)<{h}(\xi)$.

Suppose
${\langle b,Z\rangle}\le{\langle a,Y\rangle}$ forces $\xi\in\dot{g}\setminus a$.
By the property we saw in Section 2, we have $\xi\in b\setminus a$. ${\langle b,Z\rangle}$ is compatible with
${\langle b \cap (\xi+1),Y_{b\cap (\xi+1)}\rangle}$ forcing
$\dot{h}(\xi) = \eta_{b\cap(\xi+1)} < h(\xi)$, as in the proof of
Prikry lemma.
  \end{proof}

  Corollary~\ref{coro2} shows that the assumption of $A$ being progressive  
  is necessary in Theorem~\ref{progressive}~(1) and (2). 
  Corollary~\ref{coro34} does the same for Theorem~\ref{progressive}~(3).
   \begin{coro}\label{coro34}
    One can add ``$A\subsetneq \pcf{A}\subsetneq \pcf{\pcf{A}}$''
    to the list of properties of $A$ in Corollary~\ref{coro2}.
   \end{coro}
   
   \begin{proof}
   Let $A=\kappa \cap \reg$ in the final model for Corollary~\ref{coro2},
   where $2^{\kappa}$ is singular.
   By Theorem~\ref{maintheorem}~(2) we have
   $\pcf{A} = (2^{\kappa})^{+}\cap \reg = 2^{\kappa}\cap \reg \ne A$,
   which in turn implies that
   $(2^{\kappa})^{+} \in\pcf{\pcf{A}} \setminus\pcf{A}$ by Theorem~\ref{scaleexists}.
   \end{proof}
   The rest of this section is devoted to improving Corollary~\ref{coro34}. 
   Define $\operatorname{pcf}^{n}(A)$ for $n<\omega$ by 
   $\operatorname{pcf}^{0}(A) = A$ 
   and 
   $\operatorname{pcf}^{n+1}(A) = \operatorname{pcf}(\operatorname{pcf}^{n}(A))$.
   \begin{thm}\label{pcfinc}
   Suppose $\langle\kappa_{i} \mid i < \omega\rangle$ is an increasing sequence of supercompact cardinals. Then the following hold in some forcing extension:
    \begin{enumerate}
     \item[$(1)$] $\kappa_{0}$ is a singular cardinal of cofinality $\omega$.
     \item[$(2)$] $\operatorname{pcf}^{n}(\kappa_{0} \cap \reg) \subsetneq \operatorname{pcf}^{n+1}(\kappa_{0} \cap \reg)$ for every $n < \omega$.
    \end{enumerate}
   \end{thm}
   Lemma~\ref{ccpcf} ensures that sets of the form $\pcf{\theta \cap \reg}$ 
   remain the same throughout forcing extensions for Theorem~\ref{pcfinc}.
   \begin{lem}\label{ccpcf}
Suppose $A\subseteq \reg$, and
   $\mathbb{Q}$ has the $\kappa$-c.c. with $\kappa=\min (A)$.
   Then $\mathbb{Q}\force{\rm pcf}^{V}(A)\subseteq\pcf{A}$.
   \end{lem}
   \begin{proof}
    In $V$, let $\theta \in \pcf{A}$ be arbitrary. 
    Then there are an ultrafilter $D$ over $A$ and an increasing and cofinal
    sequence
    ${\langle f_{\alpha}\mid \alpha < \theta\rangle}$ in ${\left(\prod A,<_{D}\right)}$.
    Let $\dot{E}$ be a $\mathbb{Q}$-name for the filter generated by ${D}$.
 Since $\theta \geq \kappa$ remains regular after forcing with $\mathbb{Q}$,
    it suffices to prove that 
    ${\langle f_{\alpha}\mid \alpha < \theta\rangle}$ 
    is forced to be increasing and cofinal in ${\left(\prod{A},<_{\dot{E}}\right)}$.

    It is easy to see the former.
    For the latter, it suffices to prove that
    $\prod^{V}{A}$ is forced to be cofinal in ${\left(\prod{A},<_{\dot{E}}\right)}$.
Assume $p \force \dot{h} \in \prod{A}$. 
For each $\gamma \in A$, define 
$$h^*(\gamma) = \sup \{\xi + 1 \mid \exists q \leq p ( q \force \dot{h}({\gamma}) ={\xi})\}.$$
Then
 $p \force \dot{h}({\gamma}) <{h^*}({\gamma})$ for every $\gamma \in A$.  
Since $\mathbb{Q}$ has the $\kappa$-c.c. and $\gamma \geq \kappa$ is regular, 
we have $h^* \in \prod A$ in $V$, as desired.
   \end{proof}

   \begin{proof}[Proof of Theorem~\ref{pcfinc}]
    We may assume that each $\kappa_i$ 
    is indestructibly supercompact in the sense of Laver~\cite{laver} and $2^{\kappa_{i}} = \kappa_i^{+}$. We refer the reader to \cite{apter} for more details. 

    Let $\mathbb{Q}$ be the full support product
     $\prod_{i < \omega}{\rm Add}(\kappa_i,\kappa_{i + 1})$, 
     where ${\rm Add}(\kappa_{i},\kappa_{i + 1})$ is the poset 
    adding $\kappa_{i + 1}$ many Cohen subsets of $\kappa_{i}$. 
    Standard arguments show that $\mathbb{Q}$ preserves cofinalities and 
    forces $2^{\kappa_n} = \kappa_{n + 1}$ for every $n < \omega$.
     We claim that $\mathbb{Q}$ forces
     $\pcf{[\kappa_{n}^{+},\kappa_{n+1}^{+}) \cap \reg}\supseteq 
     [\kappa_{n}^{+},\kappa_{n + 2}^{+}) \cap \reg$
     for every $n < \omega$. 
    
Let $G\subseteq\mathbb{Q}$ be generic.      
    Since 
    $\mathbb{Q} \simeq \prod_{i > n}\operatorname{Add}(\kappa_{i},\kappa_{i + 1})
  \times\textstyle{\prod_{i\le n}\operatorname{Add}(\kappa_{i},\kappa_{i+1})}$ in $V$, 
  we have
     $G\simeq G_n \times H_n$ in $V[G]$.
By Theorem~\ref{maintheorem}~(1),
$\pcf{\kappa_{n}^{+} \cap \reg} 
=\pcf{\kappa_{n} \cap \reg} \cup\{\kappa_{n}\}
=(2^{\kappa_{n}})^{+} \cap \reg = \kappa_{n}^{++} \cap \reg$
in $V$.
This remains true in $V[G_n]$ 
by the $\kappa_{n+1}$-closure of the corresponding poset. 
Now we work in $V[G_n]$. 
Note that $\kappa_{n+1}$ is supercompact and  $2^{\kappa_{n+1}} = \kappa_{n + 2}$. 
By Theorem~\ref{maintheorem}~(1) we have 
$\pcf{\kappa_{n+1}^{+} \cap \reg} 
=\pcf{\kappa_{n+1} \cap \reg} \cup\{\kappa_{n+1}\}
= (2^{\kappa_{n+1}})^{+} \cap \reg = \kappa_{n + 2}^{+} \cap \reg$.
Therefore
$\pcf{[\kappa_{n}^{+},\kappa_{n+1}^{+}) \cap \reg}
= [\kappa_{n}^{+},\kappa_{n + 2}^{+}) \cap \reg$. 
Note that 
$\left(\prod_{i \le n}\operatorname{Add}(\kappa_{i},\kappa_{i+1})\right)^{V}=
\prod_{i \le n}\operatorname{Add}(\kappa_{i},\kappa_{i+1})$ 
has the $\kappa_{n}^{+}$-c.c. 
By Lemma~\ref{ccpcf} we have
$\pcf{[\kappa_{n}^{+},\kappa_{n+1}^{+}) \cap \reg}
\supseteq \pcf{[\kappa_{n}^{+},\kappa_{n+1}^{+}) \cap \reg}^{V[G_n]}
= [\kappa_{n}^{+},\kappa_{n + 2}^{+}) \cap \reg$
in $V[G]=V[G_n][H_n]$, as desired.

Since $\mathbb{Q}$ is $\kappa_{0}$-directed closed in $V$, 
    $\kappa_0$ remains supercompact in $V[G]$.
So we can define Prikry forcing $\mathbb{P}$ over $\kappa_0$. 
By Theorem~\ref{maintheorem}~(2), $\mathbb{P}$ forces
$\pcf{\kappa_{0} \cap \reg}= (2^{\kappa_{0}} )^{+}\cap \reg
    = \kappa_{1}^{+} \cap \reg$.
    By Lemma~\ref{ccpcf}, $\mathbb{P}$ forces
$[\kappa_{n}^{+},\kappa_{n + 2}^{+}) \cap \reg\subseteq 
    \pcf{[\kappa_{n}^{+},\kappa_{n+1}^{+}) \cap \reg}\subseteq 
    [\kappa_{n}^{+},(2^{\kappa_{n+1}})^{+}) \cap \reg
    =[\kappa_{n}^{+},\kappa_{n + 2}^{+}) \cap \reg$ for every $n< \omega$. 
Let $H\subseteq \mathbb{P}$ be generic. In $V[G][H]$,
we have
   ${\operatorname{pcf}^{n+1}(\kappa_0 \cap \reg)} ={\kappa_{n+1}^{+} \cap \reg}$
   by induction on $n < \omega$.
\end{proof}

   \section{An Analogue for Magidor Forcing}
   Prikry forcing is known for a wealth of variations.
   In this section, we give an analogue of Theorem~\ref{maintheorem}~(2) for one of them. 
   Here we take up Magidor forcing from \cite{magidor1978changing}, but the argument works
   equally well for other variations, e.g. the diagonal Prikry forcing as defined in \cite{NU}.

   Magidor forcing uses a sequence of ultrafilters rather than a single ultrafilter, and makes a hypermeasurable cardinal into a singular cardinal of
   uncountable cofinality. 
   For normal ultrafilters $U ,U'$ over $\kappa$, $U \vartriangleleft U'$ iff $U\in M \simeq \text{Ult}(V,U')$. Let ${\langle U_{\alpha} \mid \alpha < \lambda\rangle}$ be a $\vartriangleleft$-increasing sequence with $\lambda <\kappa$. Note that there is a such sequence if $\kappa$ is supercompact. For any $\beta< \alpha < \lambda$, we fix a function $F_{\beta}^{\alpha} \in \mbox{}^{\kappa}V$ such that $[F_{\beta}^{\alpha}]_{U_{\alpha}} = U_{\beta}$. For each $\alpha< \lambda$, define
      \begin{align*}
    A_{\alpha} & = \{\delta < \kappa \mid \forall \beta < \alpha \forall \gamma < \beta(F_{\gamma}^{\alpha}(\delta) \vartriangleleft F_{\beta}^{\alpha}(\delta) 
    \text{ are normal ultrafilters over }\delta)\}.\\
    B_{\alpha} & = \{\delta \in A_{\alpha}\setminus(\lambda + 1) \mid \forall \beta < \alpha \forall \gamma < \beta([F^{\beta}_{\gamma}\upharpoonright \delta]_{F_{\beta}^{\alpha}(\delta)} = F_{\gamma}^{\alpha}(\delta))\}.
   \end{align*}
   Note that $B_{\alpha} \in U_{\alpha}$. 
   Magidor forcing $\mathbb{M}$ is the set of pairs $\langle a,X \rangle$ such that
    \begin{itemize}
     \item $a$ is an increasing function such that
	   \begin{itemize}
	    \item $\operatorname{dom}(a) \in [\lambda]^{<\omega}$ and $\forall \alpha \in \operatorname{dom}(a)(a(\alpha) \in B_{\alpha})$.
	   \end{itemize}
     \item $X$ is a function such that
     \begin{itemize}
      \item $\operatorname{dom}(X) = \lambda \setminus \operatorname{dom}(a)$ and $\forall \alpha \in \operatorname{dom}(X)(X(\alpha) \subseteq B_{\alpha}$),
	    \item For every $\alpha \in \operatorname{dom}(X)$, if $\operatorname{dom}(a) \setminus (\alpha + 1) = \emptyset$, $X(\alpha) \in U_{\alpha}$. Otherwise, $X(\alpha) \in F^{\beta}_{\alpha}(a(\rho))$ where $\beta = \min (\operatorname{dom}(a) \setminus (\alpha + 1))$.
     \end{itemize}
\end{itemize}

    $\mathbb{M}$ is ordered by ${\langle a,X\rangle } \leq {\langle b,Y\rangle }$ iff $b \subseteq a$, $\forall \alpha \in \operatorname{dom}(X)(X(\alpha) \subseteq Y(\alpha))$ and $\forall \alpha \in \operatorname{dom}(a) \setminus \operatorname{dom}(b)(a(\alpha) \in Y(\alpha))$.
    $\mathbb{M}$ has the $\kappa^{+}$-c.c. and size $2^{\kappa}$. Thus, $\mathbb{M}$ does not change the value of $2^{\theta}$ for any $\theta \geq \kappa$. $\mathbb{M}$ preserves all cardinals above $\kappa$ but changes the cofinality of $\kappa$ like Prikry forcing. Let $\dot{g}$ be an $\mathbb{M}$-name such that $\mathbb{M} \force \dot{g} = \bigcup\{a \mid \exists X\langle a,X \rangle \in \dot{G} \}$, where $\dot{G}$ is the canonical $\mathbb{M}$-name for a generic filter. $\dot{g}$ is forced to be an increasing sequence of length $\lambda$ which converges to $\kappa$. As in Prikry forcing, we also have
   \begin{itemize}
    \item $\langle a, X \rangle \force \dot{g}\upharpoonright \operatorname{dom}({a}) = {a} \land \forall \alpha \in \lambda \setminus \operatorname{dom}(a)(\dot{g}(\alpha) \in X(\alpha))$. 
   \end{itemize}
    For each $\beta < \lambda$, We let $\mathbb{M}_{\beta}= \{{\langle a,X \rangle}_{\beta}\mid {\langle a,X \rangle} \in \mathbb{M}\}$ and $\mathbb{M}^{\beta} = \{{\langle a,X \rangle}^{\beta} \mid {\langle a,X \rangle} \in \mathbb{M}\}$. Here, ${\langle a,X\rangle}_{\beta}$ and ${\langle a,X\rangle}^{\beta}$ are ${\langle a\upharpoonright (\beta + 1),X\upharpoonright (\beta + 1)\rangle}$ and ${\langle a\upharpoonright (\lambda \setminus (\beta + 1)),X\upharpoonright (\lambda \setminus (\beta + 1))\rangle}$ respectively. The orders on $\mathbb{M}_\beta$ and $\mathbb{M}^{\beta}$ are naturally defined by that on $\mathbb{M}$.
    $\mathbb{M}$ can be factored as follows.
    \begin{lem}
     For every ${\langle a,X\rangle} \in \mathbb{M}$ and $\beta \in \operatorname{dom}(a)$, 
     we have
     $$\mathbb{M} / {\langle a,X\rangle} \simeq \mathbb{M}_{\beta}/{\langle a,X\rangle}_{\beta} \times \mathbb{M}^{\beta}/{\langle a,X\rangle}^{\beta}.$$
    \end{lem}

   Note that $\mathbb{M}_{\beta} / \langle a,X\rangle_{\beta}$ has the $a(\beta)^{+}$-c.c. 
    Lemmas~\ref{magidordiagonallemma} and \ref{magidorprikrycondition} are analogues of Lemmas~\ref{diagonallemma} and \ref{prikrycondition} for Magidor forcing respectively.
    See \cite{magidor1978changing} for proofs. 
    \begin{lem}\label{magidordiagonallemma}
     Suppose that ${\langle a,X\rangle} \in \mathbb{M}$ and $\{{\langle b,X_{b}\rangle} \mid b \in {\rm LP}\}$ is a set of extensions of $\langle a,X\rangle$ where ${\rm LP} = \{b \mid \exists Y(\langle{b,Y}\rangle \leq \langle{a,X}\rangle )\}$. Then there is a $Z$ such that $\langle a,Z \rangle \in \mathbb{M}$ and every extension of $\langle{b,Y}\rangle$ is compatible with $\langle b,X_{b}\rangle$ if $\langle b,Y \rangle \leq \langle a,Z \rangle$.
    \end{lem}
    \begin{lem}[Prikry lemma]\label{magidorprikrycondition}
     For every $\langle a,X\rangle \in \mathbb{M}$ and statement $\sigma$ of the forcing language, $\beta \in \operatorname{dom}(a)$, there is a $Z$ such that
     \begin{itemize}
      \item $\langle a,Z \rangle \leq \langle a,X \rangle$ and $\langle a,Z \rangle_{\beta} = \langle a,X \rangle_{\beta}$.
      \item If $\langle b,Y \rangle \leq \langle a,Z \rangle$ decides $\sigma$, 
      then $\langle b, Y\rangle_{\beta}^\frown \langle a,Z \rangle^{\beta}$ decides $\sigma$.
     \end{itemize}
    \end{lem}
    Here is the fundamental theorem of Magidor forcing:
    
    \begin{thm}[Magidor]
    The following hold:
     \begin{enumerate}
      \item[$(1)$] $\mathbb{M}$ adds no new subsets of $\lambda$. 
      In particular, $\lambda^{+}\cap \reg$ remains the same by $\mathbb{M}$.
      \item[$(2)$] $\mathbb{M}$ preserves all cardinals.
      \item[$(3)$] $\mathbb{M}$ forces that ${\kappa}$ is a strong limit singular cardinal of cofinality ${\lambda}$.
     \end{enumerate}
    \end{thm}
  Now we get an analogue of Theorem~\ref{maintheorem}~(2) for Magidor forcing.
  \begin{thm}\label{magidormaintheorem}
  $\mathbb{M}$ forces that $\pcf{{\kappa} \cap \reg} = (2^{{\kappa}})^{+} \cap \reg$.
  \end{thm}
  \begin{proof}
   By the proof of Theorem~\ref{maintheorem}, it suffices to show that $\mathbb{M} \force (\kappa,{(2^{\kappa})}^{+}) \cap \reg \subseteq \pcf{\kappa \cap \reg}$. Note that $(\kappa,(2^{\kappa})^{+}) \cap \reg$ remains the same after forcing with $\mathbb{M}$. Let $\theta \in (\kappa,(2^{\kappa})^{+})\cap \reg$. Let us see that $\mathbb{M} \force \theta \in \pcf{(\kappa,(2^{\kappa})^{+}) \cap \reg}$.

 For every $\gamma \leq \theta$ and $\alpha < \lambda$, we fix a function $f^{\alpha}_{\gamma} \in {^{\kappa}}\kappa$ such that $[f^{\alpha}_{\gamma}]_{U_{\alpha}} = \gamma$. We may assume $f_\theta^\alpha \in {^{\kappa}(\kappa \cap \reg)}$. Let $X' \in \prod_{\alpha < \lambda}U_\alpha$ be a function in $V$ such that $X'({\alpha}) = \{\xi \in B_\alpha \mid \forall \eta < \xi(f^{\alpha}_{\theta}(\eta) < \xi) \land \xi <f_{\theta}^{\alpha}(\xi)\}$ for any $\alpha < \lambda$.

We will show that $\mathbb{M}$ forces ${\left(\prod_{\alpha \in \lambda}{f}^{\alpha}_{\theta}(\dot{g}(\alpha)),<^{*}\right)}$ has an increasing and cofinal sequence of length $\theta$. Here, $<^{*}$ is an $\mathbb{M}$-name for the order on $\prod_{\alpha \in \lambda}{f}^{\alpha}_{\theta}(\dot{g}(\alpha))$ defined by the cobounded filter over $\lambda$. This gives the desired result, as shown by the following argument:
  
   By a usual density argument, we can find an $\mathbb{M}$-name $\dot{A}$ such that $\mathbb{M}$ forces the following properties:
\begin{itemize}
 \item $\dot{A} \in [\lambda]^{\lambda}$.
 \item $\forall \alpha,\beta \in \dot{A}(\alpha < \beta \to f_\theta^{\alpha}(\dot{g}(\alpha)) < f_\theta^{\beta}(\dot{g}(\beta)))$.
\end{itemize}
   And thus, by the proof of Theorem~\ref{maintheorem}, we have
   \begin{center}
    $\mathbb{M} \force {\left(\prod_{\alpha \in \dot{A}}{f}^{\alpha}_{\theta}(\dot{g}(\alpha)),<^{*} \upharpoonright \dot{A}\right)} \simeq \left(\prod \{{f}_{\theta}^{\alpha}(\dot{g}(\alpha)) \mid \alpha \in \dot{A}\}, <_{\dot{F}}\right)$.
   \end{center}
      Here, $\dot{F}$ is an $\mathbb{M}$-names for the cobounded filter over $\{{f}_{\theta}^{\alpha}(\dot{g}(\alpha)) \mid \alpha \in \dot{A}\}$.  It follows that $\mathbb{M}$ forces $\left(\prod \{{f}_{\theta}^{\alpha}(\dot{g}(\alpha) \mid \alpha \in \dot{A}\}, <_{\dot{F}}\right)$ has an increasing and cofinal sequence of length $\theta$. 

   For every $\gamma < \theta$, let $\dot{f}_{\gamma}$ be an $\mathbb{M}$-name for a function $\alpha \mapsto {f}^{\alpha}_{\gamma}(\dot{g}(\alpha))$. It suffices to prove 
   \begin{itemize}
    \item[(i)]
	      $\mathbb{M}\force{\langle\dot{f}_{\gamma}\mid \gamma < {\theta}\rangle}${ is increasing in }$\left(\prod_{\alpha < {\lambda}}{f}^{\alpha}_{\theta}(\dot{g}(\alpha)),<^{*}\right)$.
    \item[(ii)]
	       $\mathbb{M} \force \langle\dot{f}_\gamma \mid \gamma < \theta\rangle${ is cofinal in }$\left(\prod_{\alpha < {\lambda}}{f^{\alpha}_{\theta}}(\dot{g}(\alpha)),<^{*}\right)$.
   \end{itemize}

   (i) Let $\gamma < \delta < \theta$. Note that we have $Y(\alpha) = \{\xi < \kappa \mid f_{\gamma}^{\alpha}(\xi) < f_{\delta}^{\alpha}(\xi)\} \in U_\alpha$ for each $\alpha < \lambda$. Let $\langle a,X\rangle \in \mathbb{M}$ be arbitrary. Define $Z = (X \upharpoonright \beta_a)^{\frown}\langle X(\alpha) \cap Y(\alpha) \mid \alpha \geq \beta_a \rangle$. Here, $\beta_a = \max{\operatorname{dom}(a)}$. Then $\langle{a,Z}\rangle \leq \langle a,X\rangle$ forces $f_{\gamma}(\alpha) < f_{\delta}(\alpha)$ for every $\alpha > \beta_a$. 

 (ii) Let $\langle a,X \rangle \in \mathbb{M}$ and $\dot{h}$ be arbitrary. Suppose $\langle a,X \rangle \force \dot{h} \in \prod_{\alpha < {\lambda}}{f}^{\alpha}_{\theta}(\dot{g}(\alpha))$. By the proof of (i), we may assume that $X(\alpha) \subseteq X'(\alpha)$ for all $\alpha > \beta_b$. For each $b \in {\rm LP} = \{b \mid \exists Y(\langle b,Y\rangle \leq \langle a,X \rangle)\}$ define $Y_{b}$ and $\eta_b < \kappa$ as follows. If $\beta_b > \beta_a$, by Lemma~\ref{magidorprikrycondition} and $f_{\theta}^{\beta_b}(b(\beta_b)) < \kappa$, there is a $Y_b$ such that 
\begin{itemize}
 \item $\langle b,Y_b\rangle \leq \langle a,X\rangle$.
 \item if $\langle c,Z\rangle \leq \langle b,Y_{b}\rangle$ forces $\dot{h}(\beta_b) = {\zeta}$, then $\langle c,Z\rangle_{\beta_b}^{\frown} \langle b,Y_{b}\rangle^{\beta_b}\force \dot{h}({\beta_b}) = {\zeta}$.
\end{itemize}
   Define $\eta_{b}$ by
\begin{center}
 $\eta_{b} = \sup\{\zeta + 1< f^{\beta_b}_{\theta}(b(\beta_b))\mid \exists p \in \mathbb{M}_{\beta_b} / \langle b,Y_{b}\rangle_{\beta_b}(p^{\frown}\langle b,Y_{b}\rangle^{\beta_b} \force \dot{h}({\beta_b}) = {\zeta})\}$.
\end{center}
   Then,
\begin{center}
 $\langle b,Y_{b} \rangle \force \dot{h}({\beta_b}) < \eta_b$. 
\end{center}
   For $b \in \mathrm{LP}$ with $\beta_b \leq \beta_a$, $Y_b = X$ and $\eta_b = 0$.
 
   For each $b \in \mathrm{LP}$, since $\mathbb{M}_{\beta_b} / \langle b,Y_{b}\rangle_{\beta_b}$ has the $b(\beta_b)^{+}$-c.c., $\eta_{b} < f^{\beta_b}_{\theta}(b(\beta_b))$. For every $\alpha > \alpha_b$, define $h^{\alpha}(\xi) = \sup \{\eta_{b}< f^{\alpha}_{\theta}(\xi) \mid b \in {\rm LP} \land b(\beta_b) = \xi \land \beta_b = \alpha\}$. Because of $|\{b \in {\rm LP} \mid b(\beta_b)= \xi \land \beta_b = \alpha\}| = |\alpha| \cdot |\xi|$, we have $h^\alpha(\xi) < f_\theta^\alpha(\xi)$ for every $\xi \in X(\alpha)$. Let $\gamma = \sup_{\alpha > \beta_b}[h^{\alpha}]_{U_{\alpha}} + 1< \theta$. By Lemma~\ref{magidordiagonallemma} and the proof of (i), there is an extension $\langle a,Z \rangle \leq \langle a,X \rangle$ such that
\begin{itemize}
 \item every extension of $\langle b, Y \rangle$ is compatible with $\langle b,Y_{b}\rangle$ if $\langle b,Y \rangle \leq \langle a,Z \rangle$. 
 \item $\forall \alpha > \beta_a\forall \xi \in Z(\alpha)(h^{\alpha}(\xi) < f_{\gamma}^{\alpha}(\xi))$.
\end{itemize}

   Lastly, we claim that $\langle a,Z \rangle \force \dot{h}(\alpha) < \dot{f}_{\gamma}(\alpha)$ for all $\alpha > \beta_a$. 
   Let $\langle {b,Y}\rangle \leq \langle {a,Z} \rangle$ and $\alpha > \beta_a$ be arbitrary. Extending $\langle b,Y\rangle$ we may assume that $\alpha \in \operatorname{dom}(b) \setminus (\beta_a + 1)$. Now, we can find $\langle c,Y'\rangle$ such that $\langle b,Y\rangle \leq \langle c,Y'\rangle \leq \langle a,Z \rangle$ and $\beta_c = \alpha$. By the certain property of $\langle a,Z \rangle$, $\langle{c,Y_{c}}\rangle$ and $\langle b,Y\rangle$ have a common extension forcing $\dot{h}(\alpha) < \eta_c < h^{\alpha}(c(\alpha)) < f^{\alpha}_{\gamma}(c(\alpha)) = \dot{f}_{\gamma}(\alpha)$, as desired.
  \end{proof}
     Theorem~\ref{magidormaintheorem} enables us to generalize Theorem~\ref{pcfinc} 
     as follows, including the case of uncountable cofinality.
  \begin{thm}\label{pcfinc2}
   Suppose $\langle\kappa_{i} \mid i < \omega \rangle$ is an increasing sequence of supercompact cardinals greater than a regular cardinal $\lambda$. Then in some forcing extension the following hold:
   \begin{enumerate}
    \item[$(1)$] $\kappa_{0}$ is a singular cardinal of cofinality $\lambda$.
    \item[$(2)$] $\operatorname{pcf}^{n}(\kappa_{0} \cap \reg) \subsetneq \operatorname{pcf}^{n+1}(\kappa_{0} \cap \reg)$ for all $n< \omega$.
    \item[$(3)$] $\lambda^{+} \cap \reg = (\lambda^{+} \cap \reg)^{V}$.
   \end{enumerate}
  \end{thm}

   For $A \subseteq \reg$, define
   \begin{center}
    $\operatorname{pcf}^{\alpha}(A) = \begin{cases}A & \alpha = 0 \\ \pcf{\operatorname{pcf}^{\beta}(A)} & \alpha = \beta + 1 \\ \bigcup_{\beta < \alpha} \operatorname{pcf}^{\beta}(A) & \alpha \in {\rm Lim}
				      \end{cases}$ 
  \end{center}
Note that GCH implies $\pcf{\pcf{A}} = \pcf{A}$ for every $A \subseteq \reg$. 
By Theorem~\ref{pcfinc2}, it is consistent that 
${\langle\operatorname{pcf}^{n}(A) \mid n < \omega\rangle}$ is $\subsetneq$-increasing for some $A\subseteq \reg$. We conclude this paper with the following
\begin{ques}
   Is it a theorem of {\rm ZFC} that for every $A \subseteq \reg$ there is an $\alpha$ such that 
   $\operatorname{pcf}^{\alpha + 1}(A) = \operatorname{pcf}^{\alpha}(A)$?
  \end{ques}

   \bibliographystyle{plain}

\begin{thebibliography}{}

\end{thebibliography}


\begin{thebibliography}{10}

\bibitem{AM}
	{Abraham, U. and Magidor, M.} Cardinal arithmetic. In {Handbook of set theory. {V}ols. 1, 2, 3}, pages 1149--1227.
	Springer, Dordrecht, 2010.

 \bibitem{apter}
	 {Apter, A. W.} Some results on consecutive large cardinals. {Ann. Pure Appl. Logic} 25(1983), no.1, 1--17.

\bibitem{easton}
	{Easton, W. B.}
	Powers of regular cardinals. {Ann. Math. Logic} 1 (1970), 139--178.
	
\bibitem{gitik}
	{Gitik, M.}
	Prikry-type forcings.
	In {Handbook of set theory. {V}ols. 1, 2, 3}, pages 1351--1447.
	Springer, Dordrecht, 2010.

\bibitem{kanamori}
	{Kanamori, A.}
	{The higher infinite}. Springer Monographs in Mathematics. Springer-Verlag, Berlin, second edition, 
	Large cardinals in set theory from their beginnings, Paperback
	reprint of the 2003 edition, 2009.

 \bibitem{laver}
	 {Laver, R.}
	 Making the supercompactness of {$\kappa $} indestructible under
	 {$\kappa $}-directed closed forcing.
	 {Israel J. Math.} 29(1978), no.4, 385--388.

\bibitem{magidor1978changing}
	{Magidor, M.}
	Changing cofinality of cardinals.
	{Fund. Math.} 99(1978), no. 1, 61--71.

 \bibitem{NU}
	{Neeman, I. and Unger, S.}
	 Aronszajn trees and the {SCH}.
	 In {Appalachian set theory 2006--2012}, volume 406 of {\em London
  Math. Soc. Lecture Note Ser.}, pages 187--206. Cambridge Univ. Press,
	 Cambridge, 2013.

\bibitem{prikry}
	{Prikry, K. L.}
	Changing measurable into accessible cardinals.
	{Dissertationes Math. (Rozprawy Mat.)} 68(1970), 55.

\bibitem{shelah}
	{Shelah, S.}
	{Cardinal arithmetic}. volume~29 of {Oxford Logic Guides}.
	The Clarendon Press, Oxford University Press, New York, 
	Oxford Science Publications, 1994.

 \bibitem{silver}
	 {Silver, J.}
	 On the singular cardinals problem.
	 In {Proceedings of the {I}nternational {C}ongress of
	 {M}athematicians ({V}ancouver, {B}. {C}., 1974), {V}ol. 1}, pages 265--268, 1975.

\end{thebibliography}

\end{document}